\documentclass[12pt]{article}
\usepackage{a4}
\usepackage{amsthm}
\usepackage{amsmath}
\usepackage{amssymb}
\usepackage{amsfonts}
\usepackage{stmaryrd}
\usepackage{cite}
\usepackage{epsfig}
\newtheorem{conjecture}{Conjecture}
\newtheorem{theorem}{Theorem}
\newtheorem{lemma}[theorem]{Lemma}

\begin{document}
\title{Packing and covering directed triangles asymptotically\thanks{This work was supported by the European Research Council (ERC) under the European Union's Horizon 2020 research and innovation programme (grant agreement No 648509). The first, third and fourth authors were also supported by the MUNI Award in Science and Humanities of the Grant Agency of Masaryk University. This publication reflects only its authors' view; the European Research Council Executive Agency is not responsible for any use that may be made of the information it contains.}}

\author{Jacob W. Cooper\thanks{Faculty of Informatics, Masaryk University, Botanick\'a 68A, 602 00 Brno, Czech Republic. E-mail: {\tt jcooper@mail.muni.cz}.}\and
	Andrzej Grzesik\thanks{Faculty of Mathematics and Computer Science, Jagiellonian University, {\L}ojasiewicza 6, 30-348 Krak\'{o}w, Poland. E-mail: {\tt Andrzej.Grzesik@uj.edu.pl}.}\and
	Adam Kabela\thanks{Faculty of Informatics, Masaryk University, Botanick\'a 68A, 602 00 Brno, Czech Republic. E-mail: {\tt kabela@fi.muni.cz}.}\and
        Daniel Kr{\'a}l'\thanks{Faculty of Informatics, Masaryk University, Botanick\'a 68A, 602 00 Brno, Czech Republic, and Mathematics Institute, DIMAP and Department of Computer Science, University of Warwick, Coventry CV4 7AL, UK. E-mail: {\tt dkral@fi.muni.cz}.}}
	
\date{}
\maketitle
\begin{abstract}
A well-known conjecture of Tuza asserts that if a graph has at most $t$ pairwise edge-disjoint triangles, 
then it can be made triangle-free by removing at most $2t$ edges. 
If true, the factor $2$ would be best possible. 
In the directed setting, also asked by Tuza,
the analogous statement has recently been proven,
however, the factor $2$ is not optimal.
In this paper, 
we show that if an $n$-vertex directed graph has at most $t$ pairwise arc-disjoint directed triangles, 
then there exists a set of at most $1.8t+o(n^2)$ arcs that meets all directed triangles. 
We complement our result by presenting two constructions of large directed graphs with $t\in\Omega(n^2)$
whose smallest such set has $1.5t-o(n^2)$ arcs.
\end{abstract}

\section{Introduction}
\label{sec:intro}

The conjecture of Tuza~\cite{Tuz84,Tuz90} on packing and covering triangles in (undirected) graphs
continues to attract substantial interest within the graph theory community and
is included as one of the hundred open problems listed in the monograph by Bondy and Murty~\cite{Bon08}.
The conjecture asserts that if a graph has at most $t$ pairwise edge-disjoint triangles,
then it can be made triangle-free by removing at most $2t$ edges.
The best general result on the conjecture is a 20-year old upper bound of $2.87t$ by Haxell~\cite{Hax99}.
Yuster~\cite{Yus12} established that the conjecture holds asymptotically for dense graphs, 
building on results regarding the fractional version of the conjecture by Haxell and R\"odl~\cite{HaxR01}, and
Krivelevich~\cite{Kri95}.
While it was known that the factor $2$ is the best possible in general,
tight dense examples have recently been provided by Baron and Kahn~\cite{BarK16}.
There are also a number of results on this conjecture in particular
when restricted to special classes of graphs, e.g., \cite{ChaDMMS14,HaxKT12,Pul15,LakBT12}.

When posing the conjecture, Tuza~\cite{Tuz90} also asked about the directed version of the problem,
which is the main subject of this paper.
We will consider directed graphs without parallel arcs but possibly with bigons (two vertices connected by arcs in both ways), as well as 
directed multigraphs that may also contain parallel arcs (in addition to bigons).
For a directed (multi)graph $D$, we write $\nu_t(D)$ for the maximum number of arc-disjoint directed triangles contained in $D$ and
$\tau_t(D)$ for the minimum number of arcs whose removal renders $D$ with no directed triangles.
McDonald, Puleo and Tennenhouse~\cite{McDonPT20} answered the question of Tuza
by establishing that $\tau_t(D)\le 2\nu_t(D)-1$ for any directed multigraph $D$ and
posed the following conjecture.

\begin{conjecture}[McDonald et al.~\cite{McDonPT20}]
\label{conj}
If $D$ is a directed multigraph, then $\tau_t(D)\le 1.5\nu_t(D)$.
\end{conjecture}

Our main result is the following asymptotic improvement of the upper bound in the case of directed graphs
with the number of arc-disjoint triangles quadratic in the number of their vertices.

\begin{theorem}
\label{thm:main}
If $D$ is an $n$-vertex directed graph, then $\tau_t(D)\le 1.8\nu_t(D)+o(n^2)$.
\end{theorem}

We derive Theorem~\ref{thm:main} from Theorem~\ref{thm:fract} concerning the fractional version
of the problem (see Section~\ref{sec:prelim} for the definition of
the corresponding fractional parameters), which we prove in Section~\ref{sec:fract}.
While we can prove the fractional version for directed arc-weighted multigraphs,
the proof only yields Theorem~\ref{thm:main} in the (unweighted) directed graph setting
(see Section~\ref{sec:nfract}).

\begin{theorem}
\label{thm:fract}
If $D$ is an arc-weighted directed multigraph, then $\tau_t(D)\le 1.8\nu^*_{t}(D)$.
\end{theorem}

If true, Conjecture~\ref{conj} would be best possible.
In particular, McDonald~et~al.~\cite{McDonPT20} gave an example of a tournament $D$
with $\tau_t(D)=3$ and $\nu_t(D)=2$;
they also mention that their computational effort supports the conjecture by checking additional examples.
In Section~\ref{sec:lower}, we remark that using this tournament
one readily obtains an infinite family of tournaments attaining the factor $1.5$,
however, the number of arc-disjoint triangles in such tournaments is linear in the number of its vertices.
So, we present two constructions of directed graphs with quadratically many (in the number of vertices)
arc-disjoint triangles attaining this factor asymptotically.

\section{Preliminaries}
\label{sec:prelim}

In this section, we fix notation used throughout the paper.
If $D$ is a directed (multi)graph, we write $\lvert D\rvert$ for the size of its vertex set, and $E(D)$ for the set containing all its arcs.
Going forward, `triangle' will always mean `directed triangle', and we write $T(D)$ for the set containing all its (directed) triangles. 
In particular, if $D$ is triangle-free, it is understood to contain no directed triangles.
If $F\subseteq E(D)$,
we write $D\setminus F$ for the directed graph obtained from $D$ by removing the arcs belonging to $F$;
in the case $F=\{e\}$, we write $D\setminus e$ instead of $D\setminus\{e\}$.
An \emph{arc-weighted directed (multi)graph} is a directed (multi)graph for which every 
arc $e$ is assigned a non-negative weight $w(e)$.

A \emph{fractional triangle packing} of an arc-weighted directed multigraph $D$
is a function $m:T(D)\to[0,\infty)$ such that for every arc $uv$
the sum of $m(uvw)$ taken over all vertices $w$ such that $uvw$ is a directed triangle
is at most the weight $w(uv)$ of the arc $uv$.
The \emph{weight} of the fractional triangle packing is equal to the sum $m(uvw)$ taken over all directed triangles $uvw\in T(D)$.
The maximum weight of a fractional triangle packing of $D$ is denoted by $\nu^*_{t}(D)$.

A \emph{fractional triangle cover} of an arc-weighted directed multigraph $D$
is a function $c:E(D)\to[0,1]$ such that $c(uv)+c(vw)+c(wu)\ge 1$ for every triangle $uvw$ of $D$.
The \emph{weight} of the fractional triangle cover is equal to the sum of $w(uv)\cdot c(uv)$ taken over all arcs $uv\in E(D)$.
The minimum weight of the fractional triangle cover of $D$ is denoted by $\tau^*_{t}(D)$. 
A \emph{triangle cover} of an arc-weighted directed multigraph $D$
can be viewed as the fractional triangle cover of $D$ with a function $c:E(D)\to\{0,1\}$.
In particular, $\tau_{t}(D)$ is the minimum sum of weights of arcs,
whose removal makes a graph triangle-free,
which coincides with the definition of $\tau_{t}(D)$ for unweighted multigraphs. 
We will not use $\nu_t(D)$ for weighted multigraphs $D$.

Note that the fractional packing and cover can be viewed in terms of linear programming,
and the duality of linear programming implies that $\nu^*_{t}(D)=\tau^*_{t}(D)$
for every arc-weighted directed multigraph $D$.
In addition, every directed (multi)graph $D$ can be viewed as an arc-weighted directed (multi)graph $D$
where each arc is assigned weight one, which implies that $\nu_t(D)\le \nu^*_{t}(D)=\tau^*_{t}(D)\le\tau_t(D)$.

We conclude this section with introducing notions related to regularity partitions of directed graphs.
It will also be necessary in our exposition to use the corresponding notions for undirected graphs, however,
since both settings are completely analogous, 
we have decided to present regularity partitions only for directed graphs.
An \emph{$\varepsilon$-regular partition} of a directed graph $D$ is a partition of its vertices
into sets $V_1,\ldots,V_k$ such that 
$k\ge\varepsilon^{-1}$, 
$\left\lvert\lvert V_i\rvert-\lvert V_j\rvert\right\rvert\le 1$ for all $i,j\in [k]$, and
all but $\varepsilon k^2$ pairs of sets $V_i$ and $V_j$ are $\varepsilon$-regular;
a pair of sets $V_i$ and $V_j$, $1\le i<j\le k$, is said to be \emph{$\varepsilon$-regular}
if the following holds for every $A\subseteq V_i$ and $B\subseteq V_j$ with
$\lvert A\rvert\ge\varepsilon\lvert V_i\rvert$ and $\lvert B\rvert\ge\varepsilon\lvert V_j\rvert$:
\[\left\lvert\frac{e(A,B)}{\lvert A\rvert\cdot\lvert B\rvert}-\frac{e(V_i,V_j)}{\lvert V_i\rvert\cdot\lvert V_j\rvert}\right\rvert\le\varepsilon\mbox{ and }
  \left\lvert\frac{e(B,A)}{\lvert A\rvert\cdot\lvert B\rvert}-\frac{e(V_j,V_i)}{\lvert V_i\rvert\cdot\lvert V_j\rvert}\right\rvert\le\varepsilon ,
\]
where $e(X,Y)$ is the number of arcs oriented from $X$ to $Y$.
The directed version of Szemer\'edi regularity lemma proven by Alon and Shapira \cite{AloS04a} implies that for every $\varepsilon>0$, there exists $K(\varepsilon)$ such that
every directed graph $D$ with at least $\varepsilon^{-1}$ vertices
has an $\varepsilon$-regular partition with at most $K(\varepsilon)$ parts.
Finally, if $V_1,\ldots,V_k$ is an $\varepsilon$-regular partition of a directed graph $D$,
then the \emph{$\varepsilon$-regularity digraph} $R$ corresponding to this partition is defined as follows.
Let $R$ be the arc-weighted directed graph with $k$ vertices, with each vertex corresponding to a unique part of the partition.
If the parts $V_i$ and $V_j$ form an $\varepsilon$-regular pair,
then $R$ contains an arc from the vertex corresponding to $V_i$ to the vertex corresponding to $V_j$ with weight $\frac{e(V_i,V_j)}{\lvert V_i\rvert\cdot\lvert V_j\rvert}$, and
an arc from the vertex corresponding to $V_j$ to the vertex corresponding to $V_i$ with weight $\frac{e(V_j,V_i)}{\lvert V_i\rvert\cdot\lvert V_j\rvert}$.

\section{Fractional packing}
\label{sec:fract}

This section is devoted to the proof of Theorem~\ref{thm:fract}.
We start with a lemma that allows us to assume that an optimal fractional triangle cover
contains no arc of weight at least $5/9$.

\begin{lemma}
\label{lm:heavy-arc}
Let $D$ be an arc-weighted directed multigraph and $c:E(D)\to[0,1]$ be an optimal fractional triangle cover of $D$.
Suppose that there exists an arc $e$ such that $c(e)\ge 5/9$.
If $\tau_t(D\setminus e)\le 1.8\nu^*_{t}(D\setminus e)$, then $\tau_t(D)\le 1.8\nu^*_{t}(D)$.
\end{lemma}

\begin{proof}
Since the fractional triangle cover of $D$ restricted to $E(D)\setminus\{e\}$ is a fractional triangle cover of $D\setminus e$,
it follows that
\[\tau^*_{t}(D\setminus e)\le\tau^*_{t}(D)-c(e)w(e)\le\tau^*_{t}(D)-\frac{5}{9}w(e).\]
In particular, it holds that $1.8\tau^*_{t}(D\setminus e)+w(e)\le 1.8\tau^*_{t}(D)$.

By the assumption of the lemma,
there exists a set $F$ of arcs of weight at most $1.8\nu^*_{t}(D\setminus e)=1.8\tau^*_{t}(D\setminus e)$ such that
$F$ is a triangle cover of $D\setminus e$.
Since the set $F\cup\{e\}$ is a triangle cover of $D$ and
its weight is at most $1.8\tau^*_{t}(D\setminus e)+w(e)\le 1.8\tau^*_{t}(D)=1.8\nu^*_{t}(D)$,
the lemma follows.
\end{proof}

The next lemma is a general lemma for obtaining a triangle cover of a directed multigraph
from a triangle fractional cover. In the proof of Theorem~\ref{thm:fract},
we apply this lemma for $\alpha=5/9$ and different choices of $\beta$ and $\gamma$.

\begin{lemma}
\label{lm:triple}
Let $\alpha\ge\beta\ge\gamma\ge 0$ be three reals such that $\alpha+\beta+\gamma=1$.
Let $D$ be an arc-weighted directed multigraph and $c:E(D)\to[0,1]$ be a fractional triangle cover of $D$.
Suppose that $c(e)<\alpha$ for every arc $e$.
Further, let $A$ and $B$ be a partition of its vertex set, and
let $F$ be the set of all arcs $e$ from $A$ to $B$ with $c(e)>\beta$,
all arcs $e$ with both end vertices inside $A$ with $c(e)>\gamma$, and
all arcs $e$ with both end vertices inside $B$ with $c(e)>\gamma$.
The directed multigraph $D\setminus F$ is triangle-free.
\end{lemma}

\begin{proof}
We can assume that $\gamma < 1/3$.
Indeed, if $\gamma\ge 1/3$, then it holds that $\alpha=\beta=\gamma=1/3$.
Since $c(e)<\alpha=1/3$ for each arc $e$, it follows that $D$ is triangle-free.

We consider a triangle, say $T$, of $D$ and show that an arc of $T$ belongs to $F$.
Recall that the sum of the values assigned to arcs of $T$ by $c$ is at least one,
and that if $e$ is an arc within either the set $A$ or the set $B$ and $c(e)\ge 1/3\ge\gamma$, then $e$ belongs to $F$.
Therefore, we can assume that $T$ contains an arc from $A$ to $B$, say $e_{AB}$.
Note that the weight of the arc of $T$ contained inside either $A$ or $B$ is at most $\gamma$ and
the weight of the arc from $B$ to $A$ is less than $\alpha$.
This implies that
\[c(e_{AB})>1-\alpha-\gamma=\beta.\]
Hence, $e_{AB}$ belongs to $F$ which concludes the proof of the lemma.
\end{proof}

We are now ready to prove Theorem~\ref{thm:fract}.

\begin{proof}[Proof of Theorem~\ref{thm:fract}]
Let $c$ be an optimal fractional triangle cover of $D$ and $m$ an optimal fractional triangle packing.
By Lemma~\ref{lm:heavy-arc}, we can assume that $c(e)<5/9$ for every arc $e$ of $D$.
We will show that there exists a triangle cover with weight at most $1.8\nu^*_{t}(D)$
by presenting a random procedure for constructing a triangle cover $F$ and
showing that the expected weight of $F$ is at most $1.8\nu^*_{t}(D)$.
We remark that the procedure can be derandomized using the method of conditional expectations
to obtain a deterministic algorithm producing such a set $F$ with weight at most $1.8\nu^*_{t}(D)$.

We now describe the procedure.
First, we split the vertices into two sets $A$ and $B$,
including each vertex with probability $1/2$ to either of the sets independently of other vertices.
We next fix $\alpha$ to be $5/9$ and the values of $\beta$ and $\gamma$ as follows:
with probability $0.6$, we set $\beta=4/9$ and $\gamma=0$;
with probability $0.3$, we set $\beta=3/9$ and $\gamma=1/9$; and
with probability $0.1$, we set $\beta=\gamma=2/9$.
We refer to these three choices as Choice I, Choice II and Choice III, respectively.
Finally, we apply Lemma~\ref{lm:triple} with these values $\alpha$, $\beta$ and $\gamma$
to obtain a triangle cover $F$.

\begin{table}
\begin{center}
\begin{tabular}{|c|cccccc|}
\hline
$c(e)\in$ & $(4/9,5/9)$ & $(3/9,4/9]$ & $(2/9,3/9]$ & $(1/9,2/9]$ & $(0,1/9]$ & $0$ \\
\hline
Choice I & 0.75 & 0.50 & 0.50 & 0.50 & 0.50 & 0\\
Choice II & 0.75 & 0.75 & 0.50 & 0.50 & 0 & 0 \\
Choice III & 0.75 & 0.75 & 0.75 & 0 & 0 & 0 \\
\hline
$p(e)=$ & 0.75 & 0.60 & 0.525 & 0.45 & 0.30 & 0 \\
\hline
\end{tabular}
\end{center}
\caption{The summary of probabilities of an arc $e$ to be included in a triangle cover $F$ 
for the choices considered in the proof of Theorem~\ref{thm:fract}.}
\label{tab:prob}
\end{table}

We next analyze the expected weight of the triangle cover $F$.
Let $p(e)$ be the probability that an arc $e$ is included in the triangle cover $F$;
the value $p(e)$ depends on the value $c(e)$ only and the dependency can be found in Table~\ref{tab:prob}.
The complementary slackness implies that if $c(uv)>0$ for an arc $uv$ of $D$,
then the sum of $m(uvw)$ over all vertices $w$ such that $uvw$ is a triangle is equal to $w(uv)$.
The expected weight of the triangle cover $F$ is equal to
\begin{align*}
\sum_{uv\in E(D)}p(uv)w(uv) & = \sum_{uv\in E(D)}p(uv)\sum_{w \text{ such that } uvw\in T(D)}m(uvw) \\
                            & = \sum_{uvw\in T(D)}m(uvw)(p(uv)+p(vw)+p(wu))\\
			    & \le \max_{\substack{uvw\in T(D)\\m(uvw)>0}}(p(uv)+p(vw)+p(wu))\times\sum_{uvw\in T(D)}m(uvw).
\end{align*}
Hence, it is enough to show that $p(uv)+p(vw)+p(wu)\le 1.8$ for any triangle $uvw$ with $m(uvw)>0$.

Consider such a triangle and let $e_1$, $e_2$ and $e_3$ denote its arcs so that $c(e_1)\ge c(e_2)\ge c(e_3)$.
As $m(uvw)>0$, the complementary slackness implies that $c(e_1)+c(e_2)+c(e_3)=1$.
Observe using $c(e_1)\ge c(e_2)\ge c(e_3)$ that $p(e_1)\ge p(e_2)\ge p(e_3)$.
We distinguish four cases based on the values of $p(e_1)$ and $p(e_2)$.
\begin{itemize}
\item If $p(e_1) = p(e_2) = 0.75$, then $c(e_3) < 1/9$ and thus $p(e_3)\le 0.3$.
\item If $p(e_1) = 0.75$ and $p(e_2) = 0.6$, then $c(e_3) < 2/9$ and so $p(e_3)\le 0.45$.
\item If $p(e_1) = 0.75$ and $p(e_2) \le 0.525$, then $p(e_3)\le p(e_2)\le 0.525$.
\item If $p(e_1)\le 0.6$, then $p(e_3)\le p(e_2)\le p(e_1)\le 0.6$.
\end{itemize}
In each of the cases, it holds that $p(e_1)+p(e_2)+p(e_3)\le 1.8$;
and this concludes the proof of the theorem.
\end{proof}

\section{Non-fractional packing}
\label{sec:nfract}

In order to prove Theorem~\ref{thm:main},
we will need the next result,
which follows from~\cite[Theorem 1.2]{NutY07}
applied with ${\cal F}$ being the set containing a directed triangle;
the result can also be obtained by a reduction to its undirected analogue,
which can be derived e.g.~from~\cite[Theorem 9(ii)]{HaxR01} or \cite[Theorem 1.4]{KimKOT19}.

\begin{theorem}
\label{thm:reg}
For every $\delta>0$, there exists $\varepsilon\in(0,\delta)$ and an integer $n_0$ such that
for every directed graph $D$ with at least $n_0$ vertices the following holds:
if $R$ is the $\varepsilon$-regularity digraph of $D$ with at most $K(\varepsilon)$ vertices, then
\[\frac{\nu_t(D)}{\lvert D\rvert^2}\ge\frac{\nu^*_{t}(R)}{\lvert R\rvert^2}-\delta.\]
\end{theorem}

We now prove Theorem~\ref{thm:main}.

\begin{proof}[Proof of Theorem~\ref{thm:main}]
We need to show that for every $\delta_t>0$, there exists $n_0$ such that
every directed graph $D$ with at least $n_0$ vertices satisfies the following:
\begin{equation}
\tau_t(D)\le 1.8\nu_t(D)+\delta_t \lvert D\rvert^2.
\label{eq:m0}
\end{equation}
We apply Theorem~\ref{thm:reg} with $\delta=\delta_t/6$,
and obtain $n_0$ and $\varepsilon\le\delta_t/6$ for which Theorem~\ref{thm:reg} holds.

Let $D$ be a directed graph with at least $n_0$ vertices and
let $R$ be its $\varepsilon$-regularity digraph.
Fix a triangle cover $F_R$ of $R$ and from this define a set $F_D$ of arcs from $D$ as follows:
if $F_R$ contains
an arc from the vertex corresponding to $V_i$ to the vertex corresponding to $V_{j}$,
then $F_D$ contains all arcs from $V_i$ to $V_{j}$.
In addition, $F_D$ contains all arcs between the irregular parts and all arcs inside the parts of the partition.
Observe that $F_D$ is a triangle cover of $D$.
Since the choice of a triangle cover $F_R$ was arbitrary,
we infer that
\begin{equation*}
\frac{\tau_t(D)}{\lvert D\rvert^2}\le\frac{\tau_t(R)}{\lvert R\rvert^2}+2\varepsilon.
\end{equation*}
Thus, we derive using $\varepsilon\le\delta_t/6$ from Theorem~\ref{thm:fract} that
\begin{equation}
\frac{\tau_t(D)}{\lvert D\rvert^2}<\frac{1.8\nu^*_{t}(R)}{\lvert R\rvert^2}+\frac{\delta_t}{3}.
\label{eq:m1}
\end{equation}
On the other hand, Theorem~\ref{thm:reg} yields that
\begin{equation}
\frac{\nu_t(D)}{\lvert D\rvert^2}\ge\frac{\nu^*_{t}(R)}{\lvert R\rvert^2}-\frac{\delta_t}{3}.
\label{eq:m2}
\end{equation}
We conclude that
the inequalities \eqref{eq:m1} and \eqref{eq:m2} yield \eqref{eq:m0}.
\end{proof}

\section{Asymptotically tight constructions}
\label{sec:lower}

As noted in~\cite{McDonPT20}, the carousel tournament $D$ on five vertices
satisfies $\tau_t(D)=3$ and $\nu_t(D)=2$.
We remark that more examples of tournaments that are tight with respect to Conjecture~\ref{conj}
can be obtained by considering a transitive tournament and
replacing one or more disjoint subtournaments induced by five consecutive vertices with $D$.
Observe that the number of arc-disjoint triangles in this construction is at most linear in the number of vertices of the tournament.
In this section,
we present two constructions with the number of arc-disjoint triangles quadratic in the number of vertices that
asymptotically achieve the conjectured factor $1.5$.
The reason for presenting two different constructions is that
the structure of arc-disjoint triangles differ in the two constructions substantially,
which is evidenced by the corresponding optimal fractional covers as we discuss further.

Before presenting the first construction, we need an auxiliary lemma.
This result was proven by Brown and Harary \cite{BroH69};
since their paper is not easily accessible, we include a proof for completeness.

\begin{lemma}
\label{lm:tfree}
Let $D$ be a directed graph with $n$ vertices.
If $D$ has no directed triangle, then $D$ has at most $n^2/2$ arcs.
\end{lemma}

\begin{proof}
The proof proceeds by induction on the number of vertices of $D$.
The base case of $n\in\{1,2\}$ is clear (note that the bound is tight for $n=2$ since $D$ can be a bigon).
Suppose that $n\ge 3$.
If $D$ has no bigons, then it has at most $\binom{n}{2}\le n^2/2$ arcs and the lemma follows.
Otherwise, let $u$ and $v$ be two vertices of $D$ such that $D$ contains both an arc from $u$ to $v$ and from $v$ to $u$.
Note that no in-neighbor of $u$ is an out-neighbor of $v$ and vice-versa.
Hence,
the sum of the in-degree of $u$ and the out-degree of $v$ with respect to the vertices different from $u$ and $v$ is at most $n-2$.
Symmetrically,
the sum of the in-degree of $v$ and the out-degree of $u$ with respect to the vertices different from $u$ and $v$ is at most $n-2$.
Thus, the total number of arcs incident with $u$ or $v$ is at most $2n-2$ (including the two arcs between $u$ and $v$).
By induction, the directed graph obtained from $D$ by deleting the vertices $u$ and $v$ has at most $(n-2)^2/2$ arcs.
It follows that the number of arcs of $D$ is at most
\[\frac{(n-2)^2}{2}+2n-2=\frac{n^2-4n+4+4n-4}{2}=\frac{n^2}{2},\]
which completes the proof.
\end{proof}

We are now ready to present the first of the two constructions.

\begin{theorem}
\label{thm:random}
Let $D$ be an $n$-vertex directed graph with an arc between any pair of its vertices oriented with probability half in either direction.
It holds that $\nu_t(D)=n^2/6+o(n^2)$ and $\tau_t(D)=n^2/4-o(n^2)$ asymptotically almost surely.
\end{theorem}

\begin{proof}
We first show for every $\delta>0$ that $\nu_t(D)\ge (1/6-\delta-o(1))n^2$ asymptotically almost surely.
This follows from \cite{KeeS09} but we present a short proof for completeness.
Fix $\delta>0$ and let $\varepsilon$ be as in Theorem~\ref{thm:reg}.
Let $R$ be the $\varepsilon$-regularity digraph of $D$ with at most $K(\varepsilon)$ vertices.
Asymptotically almost surely,
$R$ is a directed graph where each pair of vertices is joined by an arc with weight $1/2-o(1)$ in each of the two directions.
Hence, it holds that $\nu^*_{t}(R)\ge (1/6-o(1))\lvert R\rvert^2$.
It follows using Theorem~\ref{thm:reg} that $\nu_t(D)\ge (1/6-\delta-o(1))n^2$ asymptotically almost surely.

We next show that $\tau_t(D)=n^2/4-o(n^2)$.
Consider any linear order on the vertex set of $D$, and
let $F$ be the set of arcs oriented from a larger vertex to a smaller one (in the considered order).
By reversing the order if necessary, we can assume that $\lvert F\rvert\le\frac{1}{2}\binom{n}{2}$.
Observe that $D\setminus F$ is triangle-free,
which implies $\tau_t(D)\le n^2/4$.
It remains to show that $\tau_t(D)\ge n^2/4-o(n^2)$ asymptotically almost surely.

Fix $\varepsilon>0$.
Let $F$ be a set of arcs such that $D\setminus F$ is triangle-free.
Consider an $\varepsilon$-regular partition of $D\setminus F$ with $k\le K(\varepsilon)$ parts and
the corresponding $\varepsilon$-regularity digraph $R$.
Let $w_R$ be the sum of the weights of the arcs of $R$.
Observe that
\[\lvert E(D\setminus F)\rvert\le \left(\frac{w_R}{k^2}+2\varepsilon+o(1)\right) n^2,\]
where the error terms correspond to the arcs inside the parts and between irregular pairs of vertex parts.
Let $R'$ be the directed graph obtained from $R$ by removing all arcs with weight at most $2\varepsilon$.
Note that the sum of the weights of the arcs of $R'$ is at least $w_R-2\varepsilon k^2$.
If $R'$ contained a triangle, then $D\setminus F$ would also contain a triangle.
Hence, $R'$ is triangle-free and consequently has at most $k^2/2$ arcs by Lemma~\ref{lm:tfree}.
Since each arc of $R$ (and so of $R'$) has weight at most $1/2+o(1)$ asymptotically almost surely,
it follows that $w_R\le (1/4+2\varepsilon+o(1))k^2$.
Consequently, we get that
\[\lvert E(D\setminus F)\rvert\le\left(\frac{1}{4}+4\varepsilon+o(1)\right) n^2,\]
which implies that
\[\lvert F\rvert\ge\left(\frac{1}{4}-4\varepsilon-o(1)\right) n^2.\]
Since the choice of $F$ was arbitrary (among the sets such that $D\setminus F$ is triangle-free),
we obtain that $\tau_t(D)\ge (1/4-4\varepsilon-o(1))n^2$ asymptotically almost surely for any $\varepsilon>0$.
It follows that $\tau_t(D)\ge n^2/4-o(n^2)$ asymptotically almost surely, which completes the proof.
\end{proof}

Our second construction is also probabilistic.
As we have mentioned,
the two constructions substantially differ in the structure of the optimal fractional triangle covers.
While all arcs in the construction presented in Theorem~\ref{thm:random}
would be assigned values close $1/3$ in an optimal fractional triangle cover,
this is not true for the construction given in the next theorem.
Here, most of the forward arcs would be assigned values close to $1/2$ and
most of the backward arcs values close to $0$.

The proof of Theorem~\ref{thm:sparse} is inspired by the construction of graphs with large girth and large chromatic number and the construction given by Baron et al.~in~\cite{BarK16}.

\begin{theorem}
\label{thm:sparse}
Let $D$ be a $n$-vertex directed graph obtained in the following way.
Firstly, for every ordered pair of vertices of $D$ insert an arc from the first vertex to the other
with probability $p=n^{-1+1/12}$; these arcs will be referred to as forward.
Next, if there is a forward arc from a vertex $u$ to a vertex $v$ and
a forward arc from a vertex $v$ to a vertex $w$, insert an arc from $w$ to $u$;
these arcs will be referred to as backward.
It holds that $\nu_t(D)=pn^2/2+o(n^{13/12})$ and $\tau_t(D)=3pn^2/4-o(n^{13/12})$ asymptotically almost surely.
\end{theorem}

\begin{proof}
Let $A$ be the set of vertices of $D$ that are contained in cycles of length at most six that
are formed by forward arcs only, regardless their orientation.
Since the expected number of such cycles is at most
\[\sum_{k=2,\ldots,6}n^k\cdot 2^k\cdot p^{k}\le 124n^{1/2},\]
it holds that $|A|=O(n^{1/2})$ asymptotically almost surely.
So, we will assume that $|A|=O(n^{1/2})$ for the remainder of the proof.
Furthermore, let $d_{\min}$ and $d_{\max}$ be $pn-n^{1/16}$ and $pn+n^{1/16}$, respectively.
The probability that the number of forward arcs
leaving a specific vertex is smaller than $d_{\min}$ or larger than $d_{\max}$
is exponentially small in $n^{1/8}/pn=n^{1/24}$;
the same holds for the number of forward arcs incoming to a specific vertex.
Hence, we can assume that the number of incoming forward arcs and
the number of outgoing forward arcs at each vertex of $D$ is between $d_{\min}$ and $d_{\max}$.
Note that this implies that each vertex of $D$ is incident
with at most $2d_{\max}^2=O(n^{1/6})$ backward arcs
since each such arc can be associated with a path formed by two forward arcs.

We first show that $\nu_t(D)\le pn^2/2+o(n^{13/12})$.
Observe that every triangle of $D$ that does not contain a vertex from $A$
consists of two forward arcs and one backward arc.
Since the number of forward arcs between vertices not contained in $A$
is at most $d_{\max}\cdot(n-\lvert A\rvert)\le d_{\max}n$,
it follows that the number of arc-disjoint triangles not containing a vertex from $A$
is at most $d_{\max}n/2=pn^2/2+O(n^{17/16})$.
Since each vertex of $A$ is incident with at most $d_{\max}$ outgoing forward arcs and
at most $d_{\max}^2$ outgoing backward arcs,
the number of arc-disjoint triangles containing a vertex from $A$ cannot exceed
\[\lvert A\rvert\cdot (d_{\max}+d_{\max}^2)=O(n^{2/3}).\]
We conclude that any set of arc-disjoint triangles in $D$
has at most $pn^2/2+O(n^{17/16})+O(n^{2/3})=pn^2/2+o(n^{13/12})$ elements asymptotically almost surely.

We next show that $\nu_t(D)\ge pn^2/2-o(n^{13/12})$ asymptotically almost surely.
Let $D'$ be a directed multigraph obtained from $D$ by removing all backward arcs,
adding a new vertex, which we will refer to as \emph{super-vertex}, and
adding arcs between the super-vertex and other vertices in such a way that
each vertex of $D$ has both in-degree and out-degree equal to $\lceil d_{\max}\rceil+1$;
note that we may need to add multiple arcs between a single vertex and the super-vertex to achieve this.
Since the directed multigraph $D'$
is strongly connected (there is a bigon between the super-vertex and each other vertex),
$D'$ contains an Eulerian tour $v_1\cdots v_K$.
We next construct a set $S$ of arc-disjoint triangles in $D$ as follows:
for every $i=1,\ldots,\lceil K/2\rceil$,
if neither of the vertices $v_{2i-1}$, $v_{2i}$ and $v_{2i+1}$ (indices modulo $K$ if needed) is the super-vertex or belong to the set $A$,
we include to $S$ the triangle formed by forward arcs $v_{2i-1}v_{2i}$ and $v_{2i}v_{2i+1}$ and
the backward arc $v_{2i+1}v_{2i-1}$.
Observe that the triangles included to $S$ are arc-disjoint (here, we use that none of the three
vertices belong to the set $A$).
It remains to estimate the size of the set $S$.
First note that $K$ is at least $n\cdot d_{\max}=pn^2+O(n^{17/16})$.
On the other hand, the number of arcs incident with the super-vertex is at most
\[2n\cdot (d_{\max}+1-d_{\min})=O(n^{17/16}),\]
and the number of arcs incident with a vertex from $A$ is at most
\[2\lvert A\rvert (d_{\max}+1)=O(n^{7/12}).\]
Hence, at least $K/2-O(n^{17/16})-O(n^{7/12})$ triples $v_{2i-1}$, $v_{2i}$ and $v_{2i+1}$
form a triangle included in the set $S$,
i.e., the set $S$ contains $pn^2/2-O(n^{17/16})=pn^2/2-o(n^{13/12})$ arc-disjoint triangles.

It remains to show that $\tau_t(D)=3pn^2/4-o(n^{13/12})$ asymptotically almost surely.
Split the vertices of $D$ randomly into two sets $X$ and $Y$,
including each vertex to either of the two sets with probability $1/2$.
Next consider the set $F$ of all forward arcs inside $X$,
all forward arcs inside $Y$, all forward arcs from $X$ to $Y$, and
all arcs incident with a vertex from $A$.
It is easy to verify that $D\setminus F$ is triangle-free.
On the other hand, the expected size of $F$ is at most
\[\frac{3}{4}d_{\max}n+2\lvert A\rvert\cdot(d_{\max}+d_{\max}^2)=
  \frac{3}{4}pn^2+O(n^{17/16})+O(n^{2/3}),
\]
since $F$ contains each forward arc with probability $3/4$ and
it also contains all arcs incident with a vertex from $A$.
Hence, asymptotically almost surely,
there is a set $F$ containing at most $3pn^2/4+O(^{17/16})=3pn^2/4+o(n^{13/12})$ arcs
such that $D\setminus F$ is triangle-free.

To finish the proof, we need to show that every set $F$ of arcs such that
$D\setminus F$ is triangle-free contains at least $3pn^2/4+o(n^{13/12})$ arcs.
Before we establish this, we identify two additional properties that
$D$ has asymptotically almost surely.
Let $X$ be a set of $\delta n\ge n/10$ vertices.
The probability that there are less than $\delta^2 pn^2-n^{51/48}$ forward arcs between the vertices of $X$
is exponential in $\frac{n^{51/24}}{\delta^2 pn^2}\ge n^{3/24}/p=\Theta(n^{25/24})$.
Since the number of such sets $X$ is at most $2^n$,
it follows that it holds asymptotically almost surely that
every set of $\delta n\ge n/10$ vertices contains at least $\delta^2 pn^2-n^{51/48}$
forward arcs between its vertices.
Let $X$ and $Y$ be two disjoint sets of vertices, each having at least $n/10$ vertices.
Similarly as in the previous case,
the probability that there are less than $p\lvert X\rvert\cdot\lvert Y\rvert-n^{51/48}$ forward arcs from $X$ to $Y$
is exponential in $\Theta(n^{25/24})$.
Since the number of such pairs of sets $X$ and $Y$ is at most $3^n$,
it holds asymptotically almost surely that
the number of forward arcs from a set $X$ of at least $n/10$ vertices
to a disjoint set $Y$ of at least $n/10$ vertices
is at least $p\lvert X\rvert\cdot\lvert Y\rvert-n^{51/48}$.
Hence, in what follows,
we will assume that the two properties given in this paragraph are true for $D$.

We now argue that every set $F$ such that
$D\setminus F$ is triangle-free contains at least $3pn^2/4+o(n^{13/12})$ arcs.
Let $F$ be such a set of arcs.
We first show that it is possible to assume that
$F$ does not contain a backward arc between two vertices
such that neither of them is contained in $A$.
If $F$ contains such an arc, say from $w$ to $u$, then there exists a vertex $v$ such that
$uv$ and $vw$ are forward arcs.
Since neither $u$ nor $w$ is contained in $A$,
the triangle $uvw$ is the only triangle containing the arc $wu$.
In particular, we can replace the arc $wu$ in $F$ with the arc $uv$
while preserving the size of $F$ and keeping the property that $D\setminus F$ is triangle-free.
Hence, we will assume that all backward arcs contained in $F$ are incident with a vertex from $A$.

Let $B$ be the set of vertices $v$ not contained in $A$ such that
$F$ contains all forward arcs from $v$ leading to vertices not contained in $A$.
Finally, let $C$ be the set of vertices not contained in $A$ or $B$.
We claim that for every vertex $v\in C$, the set $F$ contains all forward arcs to $v$
from the vertices not contained in $A$.
Otherwise, there exists a vertex $v\in C$ and a vertex $u\not\in A$ such that
$D$ contains a forward arc $uv$ and $uv\not\in F$.
Since the vertex $v$ is not included in $B$, there also exists a vertex $w\not\in A$ such that
$D$ contains a forward arc $vw$ and $vw\not\in F$.
However, the triangle $uvw$ would then contain no arc from $F$.

If the set $B$ contained at least $7n/8$ vertices,
$F$ would contain at least $p\lvert B\rvert^2-n^{51/48}\ge 3pn^2/4-n^{51/48}$ arcs
since it contains all forward arcs between vertices of $B$ by the definition of $B$.
It would follow that the size of $F$ is at least $3pn^2/4-n^{51/48}$.
Similarly, if the set $C$ contained at least $7n/8$ vertices
$F$ would contain at least $p\lvert C\rvert^2-n^{51/48}\ge 3pn^2/4-n^{51/48}$ arcs
since it must contain all forward arcs between vertices of $C$ as we have established above.
Again, the size of $F$ would be at least $3pn^2/4-n^{51/48}$.
Hence, we can assume that each of the sets $B$ and $C$
contains at least $n/8-\lvert A\rvert=n/8-O(n^{1/2})$ vertices,
in particular, at least $n/10$ elements if $n$ is large enough.
It follows that the set $F$ contains at least
\[p\left(\lvert B\rvert^2+\lvert B\rvert\cdot\lvert C\rvert+\lvert C\rvert^2\right)-O(n^{51/48})\]
arcs.
Since $\lvert B\rvert+\lvert C\rvert=n-\lvert A\rvert=n-O(n^{1/2})$,
the above lower bound on the size of $F$ is minimized if both $\lvert B\rvert$ and $\lvert C\rvert$ are $n/2-o(n)$.
We conclude that any set $F$ such that $D\setminus F$ is triangle-free
contains at least $3pn^2/4-o(n^{13/12})$ arcs.
\end{proof}

\section*{Acknowledgement}

The authors are indebted to an anonymous reviewer for their very detailed comments and
particularly for bringing the paper~\cite{NutY07} to their attention.

\bibliographystyle{bibstyle}
\bibliography{dirtuza}

\end{document}